\documentclass[reqno]{amsart}
\usepackage{amsmath, amssymb, amsfonts}
\newtheorem{theorem}{Theorem}

\newtheorem{corollary}{Corollary}
\newtheorem{lemma}{Lemma}

\theoremstyle{remark}

\theoremstyle{definition}
\newtheorem{define}{Definition}

\newtheorem*{Acknowlegement}{Acknowlegement}

\begin{document}

\title[On the CR automorphism group of a certain hypersurface ]{On the CR automorphism group of a certain hypersurface of infinite type in $\mathbb C^2$}

\author{Ninh Van Thu}

\thanks{ The research of the author was supported in part by an NRF grant 2011-0030044 (SRC-GAIA) of the Ministry of Education, The Republic of Korea.}

\address{Center for Geometry and its Applications,
 Pohang University of Science and Technology,  Pohang 790-784, The Republic of Korea}
\email{thunv@postech.ac.kr, thunv@vnu.edu.vn}

\begin{abstract}
In this article, we consider $\mathcal{C}^\infty$-smooth real hypersurfaces of infinite type in $\mathbb C^2$. The purpose of this paper is to give explicit descriptions for stability groups of the hypersurface $M(a,\alpha,p,q)$ (see Sec. $1$) and a radially symmetric hypersurface in $\mathbb C^2$.
\end{abstract}
\maketitle

\section{Introduction}
Let $M$ be a $\mathcal{C}^\infty$-smooth real hypersurface in $\mathbb C^n$ and $p\in M$. We denote by $\mathrm{Aut}(M,p)$ the stabilty group of $M$, that is, those germs at $p$ of biholomorphisms mapping $M$ into itself and fixing $p$. We also denote by $\mathrm{hol}_0(M,p)$ the set of germs at $p$ of real-analytic infinitesimal CR automorphisms of $M$ vanishing at $p$, i.e., $X\in \mathrm{hol}_0(M,p)$ if and only if there exists a germ $Z$ at $p$ of a holomorphic vector field in $\mathbb C^2$ vanishing at $p$ such that $\mathrm{Re}~Z$ is tangent to $M$ and $X=\mathrm{Re}~Z\mid_M$. 

For a real hypersurface in $\mathbb C^n$, the stability group and the real-analytic infinitesimal CR automorphism are not easy to describe explicitly; besides, it is unknown in most cases. For instance, the study of $\mathrm{Aut}(M,p)$ of various hypersurfaces is given in \cite{Ezhov, Kolar1, Kolar2, Kolar3, Kolar4, Stanton1, Stanton2}. Recently, explicit forms of the stability group of models (see detailed definition in \cite{Kolar4}) have been obtained in \cite{Ezhov, Kolar3, Kolar4}. However, these results are known for Levi nondegenerate hypersurfaces or more generally for Levi degenerate hypersurfaces of finite type.

Throughout the article, we consider $\mathcal{C}^\infty$-smooth real hypersurfaces of infinite type in $\mathbb C^2$. We shall describe the stability groups of $M(a,\alpha,p,q)$ (defined below) and a radially summetric hypersurface in $\mathbb C^2$, which are showed in \cite{Ninh, By1} that they admit nonzero tangential holomorphic vector fields vanishing at infinite type points.

Let $a(z)=\sum_{n=1}^\infty a_n z^n$ be a nonzero holomorphic function defined on $\Delta_{\epsilon_0}:=\{z\in \mathbb C\colon |z|<\epsilon_0\}~(\epsilon_0>0)$ and let $p,q$ be $\mathcal{C}^\infty$-smooth functions defined respectively on $(0,\epsilon_0)$ and $[0,\epsilon_0)$ satisfying that $q(0)=0$ and that the function
\[
g(z)=
\begin{cases}
e^{p(|z|)}&~\text{if}~0<|z|<\epsilon_0\\
0&~\text{if}~z=0 
\end{cases}
\]
is $\mathcal{C}^\infty$-smooth and vanishes to infinite order at $z=0$.

Denote by $M(a,\alpha, p,q)$ the germ at $(0,0)$ of a real hypersurface defined by
$$
\rho(z_1,z_2):= \mathrm{Re}~z_1+P(z_2)+F(z_2,\mathrm{Im}~z_1) =0,
$$
where $F$ and $P$ are respectively defined on $\Delta_{\epsilon_0}\times (-\delta_0,\delta_0)$ ($\delta_0>0$ small enough) and $\Delta_{\epsilon_0}$ by
\[
 F(z_2,t)=\begin{cases}
 -\frac{1}{\alpha}\log \Big|\frac{\cos \big(R(z_2)+\alpha t\big)}{\cos (R(z_2))} \Big| &~\text{if}~ \alpha\ne 0\\
 \tan(R(z_2))t  &~\text{if}~ \alpha =0,
\end{cases} 
 \]
 where $R(z_2)=q(|z_2|)- \mathrm{Re}\big(\sum_{n=1}^\infty\frac{a_n}{n} z_2^n\big)$ for all $z_2\in \Delta_{\epsilon_0}$,
and
\begin{equation*}
\begin{split}
  P(z_2)=
 \begin{cases}
\frac{1}{\alpha} \log \Big[ 1+\alpha P_1(z_2)\Big]~&\text{if}~ \alpha \ne 0\\
 P_1(z_2) ~&\text{if}~ \alpha=0,
\end{cases}
\end{split}
\end{equation*}
where 
\begin{equation*}
\begin{split}
P_1(z_2)=\exp\Big[p(|z_2|)+\mathrm{Re}\Big(\sum_{n=1}^\infty  \frac{a_n}{in}z_2^n\Big ) -\log \big|\cos\big(R(z_2)\big)\big|   \Big]
\end{split}
\end{equation*}
for all $z_2\in \Delta_{\epsilon_0}^*$ and $P_1(0)=0$.

We can see that $P,F$ are $\mathcal{C}^\infty$-smooth in $\Delta_{\epsilon_0}$ and $P$ vanishes to infinite order at $0$, and hence $M(a,\alpha, p,q)$ is $\mathcal{C}^\infty$-smooth and is of infinite type in the sense of D'Angelo (cf. \cite{D}). 

In \cite{Ninh}, the author proved the following theorem.
\begin{theorem}[\cite{Ninh}]\label{T1} $\mathrm{hol}_0\big(M(a,\alpha, p,q),0\big)$ is generated by
$$
H^{a,\alpha} (z_1,z_2)=L^\alpha (z_1) a(z_2)\frac{\partial }{\partial z_1}+iz_2\frac{\partial }{\partial z_2},
$$
where 
\[
L^\alpha(z_1)=
\begin{cases}
\frac{1}{\alpha}\big(\exp(\alpha z_1)-1\big)&\text{if}~ \alpha \ne 0\\
z_1 &\text{if}~ \alpha =0.
\end{cases}
\]
\end{theorem}

It is also shown in \cite{Ninh} that if $M$ is a $\mathcal{C}^\infty$-smooth hypersurface in $\mathbb C^2$ satisfying that $P$ is positive on a punctured disk, $P$ vanishes to infinite order at $0$, and $F(z_2,t)$ is real-analytic in a neighborhood of $(0,0)$ in $\mathbb C\times \mathbb R$, then $\mathrm{hol}_0(M,0)\ne 0$ if and only if, after a change of variable in $z_2$, $M=M(a,\alpha,p,q)$ for some $a,\alpha,p,q.$

We let $\phi^{a,\alpha}_t~(t\in\mathbb R)$ denote the holomorphic map defined on a neighborhood $U$ of the origin in $\mathbb C^2$ by setting
\begin{equation*}
\phi^{a,\alpha}_t(z_1,z_2)=
\begin{cases}
\Big(-\frac{1}{\alpha} \log\Big[1+(e^{-\alpha z_1}-1) \exp\big(\int_0^t a(z_2e^{i\tau})d\tau \big)\Big], z_2 e^{it}\Big)&~\text{if}~\alpha\ne 0\\
\Big(z_1 \exp\big(\int_0^t a(z_2e^{i\tau})d\tau \big), z_2 e^{it}\Big)&~\text{if}~ \alpha=0.
\end{cases}
\end{equation*}
By shrinking $U$ if necessary we can see that $\phi_t^{a,\alpha}~(t\in\mathbb R)$ is well-defined. In addition, each $\phi^{a,\alpha}_t$ preserves $M(a,\alpha,p,q)$ ( see cf. Theorem \ref{T3} in Appendix). Moreover, it is checked that $\{\phi^{a,\alpha}_t\}_{t\in\mathbb R}$ is a one-parameter subgroup of $\mathrm{Aut}\big(M(a,\alpha,p,q),0\big)$, which is generated by the holomorphic vector field $H^{a,\alpha}$. 

The first aim of this paper is to prove the following theorem.

\noindent
{\bf Theorem A.} {\it $\mathrm{Aut}\big(M(a,\alpha,p,q),0\big)=\{\phi^{a,\alpha}_t~|~ t\in \mathbb R\}$.}
 
For the case $M$ is \emph{radially symmetric}, J. Byun et al. \cite{By1} obtained the following theorem.
\begin{theorem}[\cite{By1}]\label{T2} Let $(M,0)$ be a real $\mathcal{C}^\infty$-smooth hypersurface germ at $0$ defined by the equation
$\rho(z) := \rho(z_1,z_2)=\mathrm{Re}~z_1+P(z_2)+ \mathrm{Im}~z_1 Q(z_2,\mathrm{Im}~z_1)=0$
satisfying the conditions: 
\begin{itemize}
\item[(i)] $P,Q$ are $\mathcal{C}^\infty$-smooth with $P(0)=Q(0,0)=0$,
\item[(ii)] $P(z_2)=P(|z_2|),~Q(z_2,t)=Q(|z_2|,t)$ for any $z_2$ and $t$,
\item[(iii)] $P(z_2)>0$ for any $z_2 \not= 0$, and 
\item[(iv)] $P(z_2)$ vanishes to infinite order at $z_2=0$.
\end{itemize} 
Then $\mathrm{hol}_0(M,0)=\{i\beta z_2\frac{\partial}{\partial z_2}\colon \beta\in \mathbb R\}$.
\end{theorem}

We note that the condition $(iv)$ simply says that $0$ is a point of D'Angelo infinite type. Now let us denote by $\{R_t\}_{t\in \mathbb R}$ the one-parameter subgroup of $\mathrm{Aut}(M,0)$ generated by the holomorphic vector field $H_R(z_1,z_2)=iz_2\frac{\partial }{\partial z_2}$, that is, 
$$
R_t(z_1,z_2)=\big(z_1,z_2 e^{it}\big), ~\forall t\in\mathbb R.
$$

The second aim of this paper is to show the following theorem.

\noindent
{\bf Theorem B.} {\it Let $(M,0)$ be a real $\mathcal{C}^\infty$-smooth hypersurface germ at $0$ defined by the equation
$\rho(z) := \rho(z_1,z_2)=\mathrm{Re}~z_1+P(z_2)+ \mathrm{Im}~z_1 Q(z_2,\mathrm{Im}~z_1)=0$
satisfying the conditions: 
\begin{itemize}
\item[(i)] $P,Q$ are $\mathcal{C}^\infty$-smooth with $P(0)=Q(0,0)=0$,
\item[(ii)] $P(z_2)=P(|z_2|),~Q(z_2,t)=Q(|z_2|,t)$ for any $z_2$ and $t$,
\item[(iii)] $P(z_2)>0$ for any $z_2 \not= 0$, and 
\item[(iv)] $P(z_2)$ vanishes to infinite order at $z_2=0$.
\end{itemize} 
Then $\mathrm{Aut}(M,0)=\{R_t~|~t\in \mathbb R\}$.}

This paper is organized as follows. In Section $2$, we give several properties of functions vanishing to infinite order at the origin. In Section $3$, we prove Theorem A. Section $4$ is devoted to the proof of Theorem B. Finally, a theorem is pointed out in Appendix.

\begin{Acknowlegement} The author would like to thank Prof. Do Duc Thai for his precious discussions on this material.
\end{Acknowlegement}

\section{Preliminaries}
In this section, we will recall the definition of function vanishing to infinite order at the origin in the complex plane and we will introduce several lemmas used to prove Theorem A and Theorem B.

\begin{define} We say that a $\mathcal{C}^\infty$-smooth function $P: U(0)\to \mathbb R$ on a neighborhood $U(0)$ of the origin in $\mathbb R^n$ vanishes to infinite order at $0$ if 
$$
\frac{\partial^{\alpha_1+\cdots+\alpha_n}}{\partial x_1^{\alpha_1}\cdots \partial x_n^{\alpha_n} }P(0)=0   
$$
for every index $\alpha=(\alpha_1,\ldots\alpha_n)\in \mathbb N^n$.
\end{define}
\begin{lemma}\label{lemma1} Let $P: U(0)\to \mathbb R$ be a $\mathcal{C}^\infty$-smooth function on a neighborhood $U(0)$ of the origin in $\mathbb R^n$. Then $P$ vanishes to infinite order at $0$ if and only if 
$$
\lim_{(x_1,\ldots,x_n)\to (0,\ldots,0)}\frac{P(x_1,\ldots,x_n)}{|x_1|^{\alpha_1}\cdots |x_n|^{\alpha_n}}=0
$$
for any index $\alpha=(\alpha_1,\ldots,\alpha_n)\in \mathbb N^n$.
\end{lemma}
\begin{proof}
The proof follows easily from Taylor's theorem.
\end{proof}
\begin{corollary}\label{corollary1} If a $\mathcal{C}^\infty$-smooth function $P$ on a neighborhood of the origin in $\mathbb R^n$ vanishes to infinite order at $0$, then $\frac{\partial^{\alpha_1+\cdots+\alpha_n}}{\partial x_1^{\alpha_1}\cdots \partial x_n^{\alpha_n} }P(x_1,\ldots,x_n)$ does also for any index $\alpha=(\alpha_1,\ldots\alpha_n)\in \mathbb N^n$.  
\end{corollary}
\begin{lemma}\label{lemma2} Suppose that $P\in \mathcal{C}^\infty(\Delta_{\epsilon_0})~(\epsilon_0>0)$ vanishes to infinite order at $0$, $P(z)>0$ for all $z\in \Delta_{\epsilon_0}^*$, and there are $\alpha>0$ and $\beta>0$ such that
$$
 \lim_{z\to 0}\frac{P(\alpha z)}{P(z)}=\beta.
$$ 
Then $\alpha=\beta=1$.
\end{lemma}
\begin{proof}
Suppose that there exist $\alpha>0$ and $\beta>0$ such that $ \lim_{z\to 0}\frac{P(\alpha z)}{P(z)}=\beta $. Then, we have 
$$
\frac{P(\alpha z)}{P(z)}=\beta+\gamma(z),
$$
where $\gamma$ is a function defined on $\Delta_{\epsilon_0}$ with $\gamma(z)\to 0$ as $z\to 0$. Since $\gamma(z)\to 0$ as $z\to 0$, there exists $\delta_0>0$ such that $|\gamma(z)|<\beta/2$ for any $z\in \Delta_{\delta_0}$. 

We consider the following cases.

\noindent
{\bf Case 1.} $0<\alpha<1$. In this case, fix $z_0\in \Delta_{\delta_0}^* $. Then for each positive interger $n$ we get
\begin{equation}\label{eeq1}
\begin{split}
\frac{P(\alpha^n z_0)}{P(z_0)}&=\frac{P(\alpha^n z_0)}{P(\alpha^{n-1} z_0)}\cdots \frac{P(\alpha z_0)}{P(z_0)}\\
                              &=\big(\beta+\gamma(\alpha^{n-1}z_0)\big) \cdots \big(\beta+\gamma(z_0)\big)\\
                              &\geq \big(\beta-|\gamma(\alpha^{n-1}z_0)|\big) \cdots \big(\beta-|\gamma(z_0)|\big)\\
                              &\geq  \big(\beta/2\big)^n.
\end{split}
\end{equation}
Moreover, let us choose a positive integer $m_0$ such that $\alpha^{m_0}<\beta/2$. Then it follows from (\ref{eeq1}) that
\begin{equation}\label{eeq2}
\begin{split}
\frac{P(\alpha^n z_0)}{\big(\alpha^n |z_0|\big)^{m_0}}&\geq\frac{P(z_0)}{|z_0|^{m_0}}\Big(\frac{\beta/2}{\alpha^{m_0}}\Big)^n.
\end{split}
\end{equation}
This yields that $\frac{P(\alpha^n z_0)}{(\alpha^n |z_0|)^{m_0}}\to +\infty$ as $n\to \infty$, which contradicts the fact that $P$ vanishes to infinite order at $0$. 

\noindent
{\bf Case 2.} $1<\alpha$. Since $\lim_{z\to 0}\frac{P(\alpha z)}{P(z)}=\beta$, it follows that $\lim_{z\to 0}\frac{P(\frac{1}{\alpha} z)}{P(z)}=\frac{1}{\beta}$. Following case $1$, it is impossible.

Therefore, $\alpha=1$ and thus it is obvious that $\beta=1$. The proof is complete.
\end{proof}
\begin{lemma}\label{lemma3} Let $p(t)$ be a $\mathcal{C}^\infty$-smooth function on $(0,\epsilon_0)~(\epsilon_0>0)$ such that the function
\[
P(z)=
\begin{cases}
e^{p(|z|)}&~\text{if}~z\in \Delta^*_{\epsilon_0} \\
0     &~\text{if}~z=0
\end{cases}
\]
 wanishes to infinite order at $z=0$. Let $\beta\in \mathcal{C}^\infty(\Delta_{\epsilon_0})$ with $\beta(0)=0$. Then  
$$
P(|z+z\beta(z)|)-P(|z|)=P(|z|)\Big(|z|p'(|z|)\big(\mathrm{Re}(\beta(z)+o(\beta(z))\big)\Big)+o((\beta(z))^2)
$$
for any $z\in  \Delta^*_{\epsilon_0}$ satisfying $z+z\beta(z)\in \Delta_{\epsilon_0}$.
\end{lemma}
\begin{proof}
By Taylor's theorem, for any $z\in  \Delta^*_{\epsilon_0}$ satisfying $z+z\beta(z)\in \Delta_{\epsilon_0}$ we have
\begin{equation}\label{eeq3}
\begin{split}
P(|z+z\beta(z)|)=P(|z|)+\frac{P'(|z|)}{1!}\big(|z+z\beta(z)|-|z|\big)+\frac{P''(\xi_z)}{2}\big(|z+z\beta(z)|-|z|\big)^2
\end{split}
\end{equation}
for some real number $\xi_z$ between $|z|$ and $|z+z\beta(z)|$. 

On the other hand,
\begin{equation}\label{eeq4}
\begin{split}
|z+z\beta(z)|-|z|&=\frac{|z+z\beta(z)|^2-|z|^2}{|z+z\beta(z)|+|z|}=\frac{2|z|^2\mathrm{Re}(\beta(z))+|z|^2|\beta(z)|^2}{|z+z\beta(z)|+|z|}\\
                  &= |z|\big(\mathrm{Re}(\beta(z))+o(\beta(z))\big).
\end{split}
\end{equation}
Moreover, $P'(|z|)=P(|z|)p'(|z|)$ for all $z\in \Delta^*_{\epsilon_0}$ and $P''(\xi_z)\to 0$ as $z\to 0$. Therefore, the proof follows from (\ref{eeq3}) and (\ref{eeq4}).

\end{proof}
\begin{lemma}\label{lemma4} Let $P(z)=e^{p(|z|)+g(z)}$ be a $\mathcal{C}^\infty$-smooth function on $\Delta_{\epsilon_0}~(\epsilon_0>0)$ wanishing to infinite order at $z=0$, where $g\in\mathcal{C}^\infty(\Delta_{\epsilon_0})$ and $p\in \mathcal{C}^\infty(0,\epsilon_0)$. Let $\beta\in \mathcal{C}^\infty(\Delta_{\epsilon_0})$ with $\beta(z)=O(P(z))$. Then 
$$
P(z+z\beta(z))-P(z)=P(z)\Big[|z|p'(|z|)\big(\mathrm{Re}(\beta(z))+o(\beta(z))\big)+o(\beta(z))\Big]
$$
for any $z\in  \Delta^*_{\epsilon_0}$ satisfying $z+z\beta(z)\in \Delta_{\epsilon_0}$.
\end{lemma}
\begin{proof}
Since $\beta(z)=O(P(z))$, by Lemma \ref{lemma3} we have
$$
\frac{e^{p(|z+z\beta(z)|)}}{e^{p(|z|)}}=1+|z|p'(|z|)\big(\mathrm{Re}(\beta(z))+o(\beta(z))\big)+o(\beta(z))
$$
for any $z\in  \Delta^*_{\epsilon_0}$ satisfying $z+z\beta(z)\in \Delta_{\epsilon_0}$. Then we obtain
\begin{equation*}
\begin{split}
&P(z+z\beta(z))-P(z)=P(z)\Big[\frac{e^{p(|z+z\beta(z)|)}}{e^{p(|z|)}} e^{g(z+z\beta(z))-g(z)}-1\Big]\\
&\quad=P(z)\Big[\Big(1+|z|p'(|z|)\big(\mathrm{Re}(\beta(z))+o(\beta(z))\big)+o(\beta(z))\Big) e^{g(z+z\beta(z))-g(z)}-1\Big]\\
& \quad =P(z)\Big[|z|p'(|z|)\big(\mathrm{Re}(\beta(z))+o(\beta(z))\big)+o(\beta(z))\Big]
\end{split}
\end{equation*}
for any $z\in  \Delta^*_{\epsilon_0}$ satisfying $z+z\beta(z)\in \Delta_{\epsilon_0}$. This ends the proof.
\end{proof}
\section{Stabilty group of $M(a,\alpha,p,q)$}
This section is entirely devoted to the proof of Theorem A. Let $a,\alpha,\epsilon_0,$ $\delta_0,F,$ $P,P_1,p,q$ be given as in Section  $1$. In what follows, $F$ can be written as $F(z_2,t)=tQ(z_2,t)$, where $Q$ is $\mathcal{C}^\infty$-smooth satisfying $Q(0,0)=0$. For a proof of Theorem A, we need the following lemmas.

\begin{lemma}\label{lemma5}
If $f=(f_1,f_2)\in \mathrm{Aut}\big(M(a,\alpha,p,q),0\big)$ satisfying $f_2(z_1,z_2)=\alpha z_2+\sum_{k,j=1}^\infty b_{kj} z_1^k z_2^j$, where $\alpha>0$ and $b_{kj}\in \mathbb C~(k,j\in \mathbb N^*) $, then $\alpha=1$ and $f_1(z_1,z_2)=z_1+o(z_1)$.
\end{lemma}
\begin{proof}
 Expand $f_1$ into Taylor series, we get
$$
f_1(z_1,z_2)=\sum_{k,j=0}^\infty a_{kj} z_1^k z_2^j,
$$
where $a_{jk}\in \mathbb C~(j,k\in\mathbb N)$. Note that $a_{00}=f_1(0,0)=0$.
Since $f(M(a,\alpha,p,q))\subset M(a,\alpha,p,q)$, we have
\begin{equation}\label{qt1}
\begin{split}
&\mathrm{Re}\Big(\sum_{k,j=0}^\infty a_{kj}\big(it-P(z_2)-t Q(z_2,t)\big)^k z_2^j\Big)\\
&+P\Big(\alpha z_2+ \sum_{k,j=1}^\infty b_{kj}\big(it-P(z_2)-t Q(z_2,t)\big)^k z_2^j\Big)\\
&+\mathrm{Im}\Big( \sum_{k,j=0}^\infty a_{kj}\big(it-P(z_2)-t Q(z_2,t)\big)^k z_2^j\Big)\\
&\times Q\Big(\alpha z_2+\sum_{k,j=1}^\infty b_{kj}(it-P(z_2)-t Q(z_2,t))^k z_2^j,\\
&\qquad \mathrm{Im}\big(\sum_{k,j=0}^\infty a_{kj}(it-P(z_2)-t Q(z_2,t))^k z_2^j\big)\Big)\equiv 0
\end{split}
\end{equation}
on $\Delta_{\epsilon_0}\times (-\delta_0,\delta_0)$ for some $\epsilon_0,\delta_0>0$.

We now consider the following cases.

\noindent
{\bf Case 1.} $f_1(0,z_2)\not \equiv 0$.
In this case, there is $j_1\in\mathbb N^*$ such that $a_{0j_1}\ne 0$ and $f_1(z_1,z_2)=a_{0j_1}z_2^{j_1}+o(z_2^{j_1})+O(z_1)$. 
Since $P(z_2)=o(|z_2|^{j_1})$, letting $t=0$ in (\ref{qt1}), one deduces that $\mathrm{Re}(a_{0j_1}z_2^{j_1})+ o(|z_2|^{j_1})\equiv 0$ on $\Delta_{\epsilon_0}$, which is impossible.

\noindent
{\bf Case 2.} $f_1(0,z_2)\equiv 0$. We can write $f_1(z_1,z_1)=\beta z_1+o(z_1)$, where $\beta\in \mathbb C^*$.
By (\ref{qt1}), we get
\begin{equation}\label{qqt1}
\begin{split}
&\mathrm{Re}\Big(\beta (it-P(z_2)-t Q(z_2,t))+o\big(it-P(z_2)-t Q(z_2,t)\big)\Big)\\
&+P\big(\alpha z_2+z_2O(it-P(z_2)-t Q(z_2,t))\big)\\
&+\mathrm{Im}\Big(\beta (it-P(z_2)-t Q(z_2,t))+o(it-P(z_2)-t Q(z_2,t))\Big)\\
&\times Q\Big(\alpha z_2+z_2O(it-P(z_2)-t Q(z_2,t)),\\
&\quad\quad\mathrm{Im}\big(\beta (it-P(z_2)-t Q(z_2,t))+o(it-P(z_2)-t Q(z_2,t))\big)\Big)\equiv 0
\end{split}
\end{equation}
on $\Delta_{\epsilon_0}\times (-\delta_0,\delta_0)$. In particular, inserting $z_2=0$ into (\ref{qqt1}) one has $\mathrm{Re}(\beta i)+O(t)\equiv 0$, and this thus implies $\mathrm{Im}(\beta)=0$. 

On the other hand, letting $t=0$ in (\ref{qqt1}) we obtain
\begin{equation*}
\begin{split}
-\mathrm{Re}(\beta)P(z_2)+ P\big(\alpha z_2+z_2O(P(z_2))\big)+ o(P(z_2))\equiv 0
\end{split}
\end{equation*}
on $\Delta_{\epsilon_0}$. This yields that $\lim_{z_2\to 0} P\big(\alpha z_2+z_2O(P(z_2))\big)/P(z_2)=\mathrm{Re}(\beta)> 0$. By Lemma \ref{lemma4} and the fact that $P(z_2)p'(|z_2|)$ vanishes to infinite order at $z_2=0$ (cf. Corollary \ref{corollary1}), we deduce that
 \begin{equation*}
  \lim_{z_2\to 0}\frac{P(\alpha z_2)}{P(z_2)}=\lim_{z_2\to 0} \frac{P\big(\alpha z_2+z_2O(P(z_2))\big)}{P(z_2)}=\mathrm{Re}(\beta)> 0.
 \end{equation*}
  Therefore, by Lemma \ref{lemma2} we conclude that $\alpha=\beta=1$. The proof is now complete.
\end{proof}
\begin{lemma}\label{lemma6}
If $f\in \mathrm{Aut}\big(M(a,\alpha,p,q),0\big)$ satisfying $f_1(z_1,z_2)$ $=z_1+\sum_{k=1}^\infty \sum_{j=0}^\infty $ $ a_{kj}z_1^k z_2^j$ with $a_{10}=0$ and $f_2(z_1,z_2)=z_2+\sum_{k,j=1}^\infty b_{kj}z_1^k z_2^j$, where $a_{kj},b_{kj}\in \mathbb C~(k,j\in \mathbb N)$, then $f=id$.
\end{lemma}
\begin{proof}
Since $f$ preserves $M(a,\alpha,p,q)$, it follows that
\begin{equation}\label{qqt2}
\begin{split}
&\mathrm{Re}\Big(\big(it-P(z_2)-t Q(z_2,t)\big)+\sum_{k=1}^\infty\sum_{j=0}^\infty  a_{kj}\big(it-P(z_2)-t Q(z_2,t)\big)^k z_2^j\Big)\\
&+P\Big(z_2+\sum_{k,j=1}^\infty  b_{kj}\big(it-P(z_2)-t Q(z_2,t)\big)^k z_2^j\Big)\\
&+\mathrm{Im}\Big(\big(it-P(z_2)-t Q(z_2,t)\big)+\sum_{k=1}^\infty \sum_{j=0}^\infty a_{kj}\big(it-P(z_2)-t Q(z_2,t)\big)^k z_2^j\Big)\\
&\times Q\Big(z_2+\sum_{k,j=1}^\infty b_{kj}\big(it-P(z_2)-t Q(z_2,t)\big)^k z_2^j,\\
&\qquad \mathrm{Im}\Big(\big(it-P(z_2)-t Q(z_2,t)\big)+\sum_{k=1}^\infty\sum_{j=0}^\infty  a_{kj}\big(it-P(z_2)-t Q(z_2,t)\big)^k z_2^j\Big)\Big)\equiv 0,
\end{split}
\end{equation}
or equivalently,
\begin{equation}\label{qqt3}
\begin{split}
&\mathrm{Re}\Big(\sum_{k=1}^\infty \sum_{j=0}^\infty a_{kj}\big(it-P(z_2)-t Q(z_2,t)\big)^k z_2^j\Big)\\
&+P\Big(z_2+\sum_{k,j=1}^\infty  b_{kj}\big(it-P(z_2)-t Q(z_2,t)\big)^k z_2^j\Big)-P(z_2)\\
& +t \Big[ Q\Big(z_2+\sum_{k,j=1}^\infty  b_{kj}\big(it-P(z_2)-t Q(z_2,t)\big)^k z_2^j,\\
&\qquad\quad t+\mathrm{Im}\big(\sum_{k=1}^\infty \sum_{j=0}^\infty a_{kj}\big(it-P(z_2)-t Q(z_2,t)\big)^k z_2^j\big)\Big)-Q(z_2,t)\Big]\\
&+\mathrm{Im}\Big(\sum_{k=1}^\infty \sum_{j=0}^\infty a_{kj}\big(it-P(z_2)-t Q(z_2,t)\big)^k z_2^j\Big)\\
&\times Q\Big(z_2+\sum_{k,j=1}^\infty  b_{kj}\big(it-P(z_2)-t Q(z_2,t)\big)^k z_2^j,\\
&\quad \qquad t+\mathrm{Im}\big(\sum_{k=1}^\infty \sum_{j=0}^\infty a_{kj}\big(it-P(z_2)-t Q(z_2,t)\big)^k z_2^j\big)\Big)\equiv 0
\end{split}
\end{equation}
on $\Delta_{\epsilon_0}\times (-\delta_0,\delta_0)$ for some $\epsilon_0,\delta_0>0$.

If $f_1(z_1,z_2)\equiv z_1$, then let $k_1=+\infty$. In the contrary case, let $k_1$ be the smallest integer $k$ such that $a_{kj}\ne 0$ for some $j\in \mathbb N^*$. Then let $j_1$ be the smallest integer $j$ such that $a_{k_1j}\ne 0$. Similarly, if $f_2(z_1,z_2)\equiv z_2$, then denote by $k_2=+\infty$. Otherwise, let $k_2$ be the smallest integer $k$ such that $b_{kj}\ne 0$ for some $j\in \mathbb N^*$. Denote by $j_2$ the smallest integer $j$ such that $b_{k_2j}\ne 0$. 

Since $P(z_2)=o(|z|^j)$ for any $j\in \mathbb N$, inserting $t=\alpha P(z_2)$ into (\ref{qqt3})
(with $\alpha \in \mathbb R$ to be chosen later) one gets
\begin{equation}\label{qqt4}
\begin{split}
&\mathrm{Re} \Big(a_{k_1j_1}P^{k_1}(z_2)(\alpha i-1)^{k_1}\big( z_2^{j_1}+o(|z_2|^{j_1})\big)\Big)\\
&+P\Big(z_2+b_{k_2j_2}P^{k_2}(z_2) (\alpha i-1)^{k_2}\big(z_2^{j_2}+o(|z_2|^{j_2})\big)\Big)-P(z_2)\\
& +\alpha P(z_2) \Big[ Q\Big(z_2+b_{k_2j_2}P^{k_2}(z_2)(\alpha i-1)^{k_2}\big(z_2^{j_2}+o(|z_2|^{j_2})\big),\\
&\qquad \alpha P(z_2)+\mathrm{Im}\big(a_{k_1j_1}P^{k_1}(z_2)(\alpha i-1)^{k_1}\big( z_2^{j_1}+o(|z_2|^{j_1})\big)\big)\Big)-Q(z_2,\alpha P(z_2))\Big]\\
&+\mathrm{Im}\Big(a_{k_1j_1}P^{k_1}(z_2)(\alpha i-1)^{k_1}\big( z_2^{j_1}+o(|z_2|^{j_1})\big)\Big)\\
&\times Q\Big(z_2+b_{k_2j_2}P^{k_2}(z_2)(\alpha i-1)^{k_2}\big(z_2^{j_2}+o(|z_2|^{j_2})\big),\\
&\qquad \alpha P(z_2)+\mathrm{Im}\big(a_{k_1j_1}P^{k_1}(z_2)(\alpha i-1)^{k_1} \big( z_2^{j_1}+o(|z_2|^{j_1})\big)\big)\Big)\equiv 0
\end{split}
\end{equation}
on $\Delta_{\epsilon_0}$. Since $Q(0,0)=0$, (\ref{qqt4}) tells us that 
\begin{equation}\label{qqt4,5}
\begin{split}
&P\Big(z_2+b_{k_2j_2}P^{k_2}(z_2) (\alpha i-1)^{k_2}\big(z_2^{j_2}+o(|z_2|^{j_2})\big)\Big)-P(z_2)\\
&\quad +\mathrm{Re} \Big(a_{k_1j_1}P^{k_1}(z_2)(\alpha i-1)^{k_1}\big( z_2^{j_1}+o(|z_2|^{j_1})\big)\Big)+ P^{k_2+1}(z_2)o(|z_2|^{j_2})\equiv 0
\end{split}
\end{equation}
on $\Delta_{\epsilon_0}$. Moreover, one has by Lemma \ref{lemma4} that 
\begin{equation}\label{qqt4,51}
\begin{split}
&P^{k_2+1}(z_2)\Big[ |z_2|p'(|z_2|)\Big(\mathrm{Re} \big(b_{k_2j_2}(\alpha i-1)^{k_2}z_2^{j_2-1}\big)+o(|z_2|^{j_2-1})\Big)+o(|z_2|^{j_2-1})\Big]\\
&+P^{k_1}(z_2)\mathrm{Re} \Big[a_{k_1j_1}(\alpha i-1)^{k_1}\big( z_2^{j_1}+o(|z_2|^{j_1})\big)\Big]\equiv 0
\end{split}
\end{equation}
on $\Delta_{\epsilon_0}$.

We now observe that $\limsup_{r\to 0^+} |r p'(r)|=+\infty$, for otherwise one gets $|p(r)|\lesssim |\log(r)|$ for every $0<r<\epsilon_0$, and thus $P$ does not vanish to infinite order at $0$. We thus divide the proof into two cases as follows.

\noindent
{\bf Case 1.}  $k_2<+\infty$ and $k_2+1< k_1\leq +\infty$. 
In this case, $P^{k_1}(z_2)=o(P^{k_2+1}(z_2))$ and hence (\ref{qqt4,51}) yields
\begin{equation}\label{qqt4,52}
\begin{split}
&P^{k_2+1}(z_2)\Big[ |z_2|p'(|z_2|)\Big(\mathrm{Re} \big(b_{k_2j_2}(\alpha i-1)^{k_2}z_2^{j_2-1}\big)+o(|z_2|^{j_2-1})\Big)\Big]\equiv o(P^{k_2+1}(z_2))\\
\end{split}
\end{equation}
on $\Delta_{\epsilon_0}$. It is absurd.
 
\noindent
{\bf Case 2.}  $k_1<+\infty$ and $k_1-1\leq k_2\leq +\infty$. 

By the fact that $P(z_2)p'(|z_2|)$ vanishes to infinite order at $z_2=0$ (see Corollary \ref{corollary1}), Lemma \ref{lemma4}, and  Eq. (\ref{qqt4,51}), it follows that 
\begin{equation}\label{qqt7}
\begin{split}
&\mathrm{Re} \Big(a_{k_1j_1}((\alpha i-1)^{k_1} z_2^{j_1}\Big)+o(|z_2|^{j_1})\equiv 0
\end{split}
\end{equation}
on $\Delta_{\epsilon_0}$. Notice that if $j_1=0$, then $k_1\geq 1$ and $\alpha$ can thus be chosen so that $\mathrm{Re}(a_{k_1j_1}(\alpha i-1)^{k_1})\ne 0$. Therefore, Eq. (\ref{qqt7}) is a contradiction. 

\noindent
{\bf Case 3.} $k_2+1= k_1\leq +\infty$. 

Since $k_2+1= k_1<+\infty$, we have by (\ref{qqt4,51})
 \begin{equation}\label{qqt5}
\begin{split}
&h(z_2):= |z_2|p'(|z_2|)\Big(\mathrm{Re} \big(b_{k_2j_2}(\alpha i-1)^{k_2}z_2^{j_2-1}\big)+o(|z_2|^{j_2-1})\Big)\\
&+\mathrm{Re} \Big[a_{k_1j_1}(\alpha i-1)^{k_1}\big( z_2^{j_1}+o(|z_2|^{j_1})\big)\Big]+o(|z_2|^{j_2-1})\equiv 0
\end{split}
\end{equation}
on $\Delta_{\epsilon_0}$. Because $\limsup_{r\to 0^+}r p'(r)=+\infty$, $j_2-1=j_1+d$ for some $d\in \mathbb N^*$. Theorefore taking $lim_{r\to 0^+} \frac{1}{r^{j_1}}h(r e^{i\theta})$ for each $\theta\in\mathbb R$, from (\ref{qqt4}) one obtains
$$
\mathrm{Re}\big(c_1 e^{i(j_1+d)\theta}\big)= \mathrm{Re}\big(c_2 e^{ij_1\theta}\big)
$$
for every $\theta\in \mathbb R$, where $c_1,c_2\in \mathbb C^*$. This is impossible since $\{ 1, \cos \theta,\sin \theta, \ldots, \cos ((j_1+d)\theta), \sin ((j_1+d)\theta)\}$ are linearly independent.

Altogether, we conclude that $k_1=k_2=+\infty$, and hence the proof is complete.
\end{proof}

Now we are ready to prove Theorem A.
\begin{proof}[Proof of Theorem A]
For $f=(f_1,f_2)\in \mathrm{Aut}\big(M(a,\alpha,p,q),0\big)$, we let $\{F_t\}_{t\in \mathbb R}$ be the family of automorphisms by setting $F_t:=f\circ \phi^{a,\alpha}_{-t}\circ f^{-1}$. Then it follows that $\{F_t\}_{t\in \mathbb R}$ is a one-parameter subgroup of $\mathrm{Aut}\big(M(a,\alpha,p,q),0\big)$. By Theorem \ref{T1}, there exists a real number $\delta$ such that $F_t=\phi^{a,\alpha}_{\delta t}$ for all $t\in \mathbb R$. This implies that
\begin{equation}\label{tte1}
f= \phi^{a,\alpha}_{\delta t}\circ f\circ \phi^{a,\alpha}_{t},~\forall t\in \mathbb R .
\end{equation}
We note that if $\delta=0$, then $f=f\circ \phi^{a,\alpha}_t$ and thus $\phi^{a,\alpha}_t=id$ for any $t\in \mathbb R$, which is a contradiction. Hence, we may assume that $\delta\ne 0$.

Now we shall prove that $\delta=-1$. Indeed, we have by (\ref{tte1})
\begin{equation}\label{tte2}
\begin{split}
 f_2(z_1,z_2)\equiv e^{i\delta t} f_2\big(z_1\exp(\int_0^t a(z_2e^{i\tau})d\tau), z_2 e^{it}\big)
 \end{split}
\end{equation}
on a neighborhood $U$ of $(0,0)\in \mathbb C^2$ and for all $t\in \mathbb R$.

Expand $f_2$ into Taylor series, one obtains that
$$
f_2(z_1,z_2)=\sum_{k,j=0}^\infty b_{kj} z_1^kz_2^j,
$$
where $b_{kj}\in \mathbb C~(k,j\in \mathbb N)$ and $b_{00}=f_2(0,0)=0$. Hence, Eq. (\ref{tte2}) is equivalent to
\begin{equation}\label{tte3}
\begin{split}
\sum_{k,j=0}^\infty b_{kj} z_1^kz_2^j\equiv \sum_{k,j=0}^\infty b_{kj} z_1^kz_2^j \exp\Big(i(j+\delta)t+k\int_0^t a(z_2e^{i\tau})d\tau\Big)
\end{split}
\end{equation}
on $U$ for all $t\in \mathbb R$. Taking the derivative both sides of (\ref{tte3}) with respect to $t$, we arrive at
\begin{equation}\label{tte4}
\begin{split}
\sum_{k,j=0}^\infty b_{kj} z_1^kz_2^j\Big(i(j+\delta)+k a(z_2e^{it})\Big) \exp\Big(i(j+\delta)t+k\int_0^t a(z_2e^{i\tau})d\tau\Big)\equiv 0
\end{split}
\end{equation}
on $U$ for all $t\in \mathbb R$. Moreover, letting $z_2=0$ in (\ref{tte4}) one has
$$
\sum_{k=1}^\infty i\delta b_{k0} z_1^k \equiv 0
$$
on the set $\{z_1\in \mathbb C~|~(z_1,0)\in U\}$. This yields that $b_{k0}=0$ for every $k=1,2,\ldots$. Besides, since $f$ is a biholomorphism we get $b_{01}\ne 0$.

On the other hand, letting $z_1=0$ and $t=0$ in (\ref{tte4}) we conclude that
\begin{equation}\label{e4}
\begin{split}
\sum_{j=1}^\infty b_{0j} z_2^j\Big(i(j+\delta)\Big)\equiv 0
\end{split}
\end{equation}
on $\{z_2\in \mathbb C~|~(0,z_2)\in U\}$. Since $b_{01}\ne 0$, (\ref{e4}) entails that $\delta=-1$ and furthermore $b_{0j}=0$ for all $j=2,3,\ldots$. In addition, replacing $f$ by $f\circ \phi^{a,\alpha}_\theta$ for a resonable $\theta \in \mathbb R$, we can assume that $b_{01}=\alpha>0$, and thus $f_2(z_2)=\alpha z_2+\sum_{k,j=1}^\infty b_{kj}z_1^kz_2^j$.

Applying Lemma \ref{lemma5}, we conclude that $f_1(z_1,z_2)=z_1+o(z_1)$ and $f_2(z_1,z_2)=z_2+O(z_1 z_2)$. Finally, Lemma \ref{lemma6} ensures that $f=id$, and thus the proof is complete
\end{proof}
\section{Stability groups of radially symmetric hypersurfaces of infinite type}
In this section, we are going to prove Theorem B. To do this, let $M$ be a $\mathcal{C}^\infty$-smooth hypersurface as in Theorem B. That is, $M$ is defined by 
$$
\rho(z_1,z_2)=\mathrm{Re}~z_1+P(z_2)+\mathrm{Im}~z_1 Q(z_2,\mathrm{Im}~z_1)=0,
$$
where $P,Q$ are $\mathcal{C}^\infty$-smooth functions on $\Delta_{\epsilon_0}$ and $\Delta_{\epsilon_0}\times (-\delta_0,\delta_0)~(\epsilon_0,\delta_0>0)$, respectively, satisfying conditions $(i)-(iv)$ as in Theorem B.

 In order to prove Theorem B, we need the following lemma.
\begin{lemma}\label{lemma7}
If $f\in \mathrm{Aut}(M,0)$ satisfying $f_1(z_1,z_2)=\sum_{k=1}^\infty a_k z_1^k$ and $f_2(z_1,z_2)=z_2\sum_{j=0}^\infty b_j z_1^j$, where $a_k,b_j\in \mathbb C~(k,j\in \mathbb N)$, $b_0>0$ and $a_1\ne 0$, then $a_1=b_0=1$
\end{lemma}
\begin{proof}
Since $M$ is invariant under $f$, we have
\begin{equation}\label{qtt1}
\begin{split}
&\mathrm{Re}\Big(\sum_{k=1}^\infty a_k\big(it-P(z_2)-t Q(z_2,t)\big)^k\Big)\\
&+P\Big(z_2 \sum_{j=0}^\infty b_{j}\big(it-P(z_2)-t Q(z_2,t)\big)^j\Big)\\
&+\mathrm{Im}\Big(\sum_{k=1}^\infty a_k(it-P(z_2)-t Q(z_2,t))^k\Big)\\
&\times Q\Big(z_2 \sum_{j=0}^\infty b_{j}(it-P(z_2)-t Q(z_2,t))^j,\\
&\qquad \quad \mathrm{Im}\big(\sum_{k=1}^\infty a_k(it-P(z_2)-t Q(z_2,t))^k\big)\Big)\equiv 0
\end{split}
\end{equation}
on $\Delta_{\epsilon_0}\times (-\delta_0,\delta_0)$. It follows from (\ref{qtt1}) with $z_2=0$ that 
$$
\mathrm{Re}(a_1 it)+o(t)= 0
$$
for every $t\in\mathbb R$ small enough. This yields that $\mathrm{Im}(a_1)=0$.

On the other hand, inserting $t=0$ into (\ref{qtt1}) one has 
\begin{equation}\label{qtt2}
\begin{split}
P\Big(b_0z_2+z_2O(P(z_2)) \Big)-\mathrm{Re}(a_1)P(z_2)+o(P(z_2))\equiv 0
\end{split}
\end{equation}
on $\Delta_{\epsilon_0}$. This implies that $\lim_{z_2\to 0} P\big(b_0 z_2+z_2O(P(z_2))\big)/P(z_2)=\mathrm{Re}(a_1)=a_1> 0$. 

By assumption, we can write $P(z_2)=e^{p(|z_2|)}$ for all $z_2\in \Delta^*_{\epsilon_0}$ for some function $p\in\mathcal{C}^\infty(0,\epsilon_0) $ with $\lim_{t\to 0^+}p(t)=-\infty$ such that $P$ vanishes to infinite order at $z_2=0$. Therefore, by Lemma \ref{lemma4} and the fact that $P(z_2)p'(|z_2|)$ vanishes to infinite order at $z_2=0$ (cf. Corollary \ref{corollary1}), one gets that
 $$
  \lim_{z_2\to 0}\frac{P(b_0 z_2)}{P(z_2)}=\lim_{z_2\to 0} \frac{P\big(b_0 z_2+z_2O(P(z_2))\big)}{P(z_2)}=a_1> 0.
 $$
Hence, Lemma \ref{lemma2} ensures that $a_1=b_0=1$, which ends the proof.
\end{proof}

\begin{proof}[Proof of Theorem B]
For $f=(f_1,f_2)\in \mathrm{Aut}(M,0)$. We define $F_t$ by setting $F_t:=f\circ R_{-t}\circ f^{-1}$ for each $t\in \mathbb R$. Then $\{F_t\}_{t\in \mathbb R}$ is a one-parameter subgroup of $\mathrm{Aut}(M,0)$. 

Using the same arguments as in the proof of Theorem A, Theorem \ref{T2} yields that $F_t=R_{-t}$ for all $t\in \mathbb R$. This implies that
\begin{equation}\label{e1}
f= R_{-t}\circ f\circ R_{t},~\forall t\in \mathbb R,
\end{equation}
namely
\begin{equation*}
\begin{cases}
f_1(z_1,z_2)&\equiv f_1(z_1,z_2 e^{it})\\
f_2(z_1,z_2)&\equiv e^{-it} f_2(z_1,z_2 e^{it})
\end{cases}
\end{equation*}
on a neighborhood $U$ of $(0,0)$ in $\mathbb C^2$ for all $t\in \mathbb R$. Indeed, this tells us that $f_1(z_1,z_2)=\sum_{k=1}^\infty a_k z_1^k$ and $f_2(z_1,z_2)=z_2\sum_{j=0}^\infty b_j z_1^j$ for all $(z_1,z_2)\in U$, where $a_k, b_j\in \mathbb C$ for all $j\in \mathbb N$ and $k\in \mathbb N^*$. We note that $a_1, b_0\in \mathbb C^* $. In addition, replacing $f$ by $f\circ R_\theta$ for some $\theta\in \mathbb R$, we can assume that $b_0$ is a positive real number. 

We now apply Lemma \ref{lemma7} to obtain that $a_1=b_0=1$. Finally, by  Lemma \ref{lemma6} we conclude that $f=id$. ( Lemma \ref{lemma6} still holds for a $\mathcal{C}^\infty$-smooth radially symmetric hypersurface satisfying $(i)-(iv)$.) Hence, the proof is complete.
\end{proof}
\section{Appendix}
\begin{theorem} \label{T3} Let $p_0\in M(a,\alpha,p,q)$. Then any flow of the holomorphic vector field
$$
H^{a,\alpha}(z_1,z_2)=L^\alpha(z_1) a(z_2)\frac{\partial }{\partial z_1}+i z_2\frac{\partial }{\partial z_2},
$$
where 
\begin{equation*}
L^\alpha(z_1)=
\begin{cases}
\frac{1}{\alpha}\big(\exp(\alpha z_1)-1\big)&\text{if}~ \alpha \ne 0\\
z_1 &\text{if}~ \alpha =0,
\end{cases}
\end{equation*}
starting from $p_o$ is contained in $M(a,\alpha,p,q)$.
\end{theorem}
\begin{proof}
Let $P_1, P, R, F$ be functions and $\epsilon_0>0,\delta_0>0$ be positive real numbers introduced to define $M(a,\alpha,p,q)$ and let $Q_0(z_2):=\tan (R(z_2))$ for all $z_2 \in \Delta_{\epsilon_0}$. Then by Lemma $7$, Lemma $8$, and Corollary $9$ in \cite[Appendix A]{Ninh} we have the following equations.
\begin{align*}
&(i)\;\mathrm{Re}\Big[i z_2 {Q_0}_{z_2}(z_2)+\frac{1}{2}\Big(1+Q_0^2(z_2)\Big)i a(z_2)\Big]\equiv 0;\\
&(ii)\; \mathrm{Re}\Big[i z_2 {P_1}_{z_2}(z_2)-\Big(\frac{1}{2}+\frac{Q_0(z_2)}{2i}\Big)a(z_2)P_1(z_2)\Big]\equiv 0;\\
&(iii)\;\mathrm{Re}\Big[i z_2 {P}_{z_2}(z_2)+\frac{\exp\big(-\alpha P(z_2)\big)-1}{\alpha}\left(\frac{1}{2}+\frac{Q_0(z_2)}{2i}\right) a(z_2)\Big]\equiv 0\quad\text{for}\;\alpha\ne 0;\\
&(iv)\;\Big(i+F_t(z_2,t\Big)\exp\Big(\alpha\big(i t-F(z_2,t)\big)\Big)\equiv i+Q_0(z_2);\\
&(v)\;\mathrm{Re}\Big[2i\alpha  z_2 F_{z_2}(z_2,t)+\Big(F_t(z_2,t)-Q_0(z_2)\Big) ia(z_2)\Big]\equiv 0\\
\end{align*}
on $\Delta_{\epsilon_0}$ for any $t\in (-\delta_0,\delta_0)$.

Let $z(t)=(z_1(t),z_2(t)), -\infty<t<+\infty$, be the flow of $H$ satisfying $z(0)=p_0$. This means that 
\begin{equation*}
\begin{cases} 
z_1'(t)=L(z_1(t))a(z_2(t))\\
z_2'(t)=i z_2(t)
\end{cases}
\end{equation*}
for all $t\in\mathbb R$.

Let $g(t):=\rho(z_1(t),z_2(t)), -\infty<t<+\infty $. Then 
$$
g'(t)=2\mathrm{Re}\Big [ \rho_{z_1}(z(t))z_1'(t)+\rho_{z_2}(z(t))z_2'(t)\Big],~ \forall t\in \mathbb R.
$$

 We devide the proof into two cases.
 
\noindent 
 {\bf a) $\alpha =0$.} In this case, $F(z_2,\tau)=Q_0(z_2)\tau$ for all $(z_2,\tau)\in \Delta_{\epsilon_0}\times (-\delta_0,\delta_0)$. Therefore, by $\mathrm{(i)}$ and $\mathrm{(ii)}$ one obtains that
\begin{equation*}
\begin{split}
g'(t)&=2\mathrm{Re}\Big[\Big(\frac{1}{2}+\frac{Q_0(z_2(t))}{2i}\Big) z_1(t)a(z_2(t))\\
&\quad +\Big({P_1}_{z_2}(z_2(t))+(\mathrm{Im}~z_1(t)){Q_0}_{z_2}(z_2(t))\Big) i\beta z_2(t)\Big]\\
&= 2\mathrm{Re}\Big[\Big(\frac{1}{2}+\frac{Q_0(z_2(t))}{2i}\Big) \Big(i (\mathrm{Im}~z_1(t))+g(t)-P_1(z_2(t))\\
&\quad -(\mathrm{Im}~z_1) Q_0(z_2(t))\Big)a(z_2(t))+\Big({P_1}_{z_2}(z_2(t))+(\mathrm{Im}~z_1){Q_0}_{z_2}(z_2(t))\Big) i z_2(t)\Big]\\
&=2\mathrm{Re}\Big[i z_2(t) {P_1}_{z_2}(z_2(t)) -\Big(\frac{1}{2}+\frac{Q_0(z_2(t))}{2i}\Big)a(z_2(t)) P_1(z_2(t))\Big]\\
&\quad +(\mathrm{Im}~z_1(t))\mathrm{Re}\Big[i  z_2(t) {Q_0}_{z_2}(z_2(t)) +\frac{1}{2}\Big(1+Q_0(z_2(t))^2\Big)ia(z_2(t))\Big]\\
&\quad +  2g(t)\mathrm{Re}\Big[\Big(\frac{1}{2}+\frac{Q_0(z_2(t))}{2i}\Big) a(z_2(t))\Big ]\\
&=  2g(t)\mathrm{Re}\Big[\Big(\frac{1}{2}+\frac{Q_0(z_2(t))}{2i}\Big) a(z_2(t))\Big ]
\end{split}
\end{equation*}
 for every $t\in \mathbb R$. Since $g(0)=\rho(p_0)=0$, by the uniqueness of the solution of differential equations, we conclude that $g(t)\equiv 0$. This proves the theorem for $\alpha =0$.
  
\noindent
{\bf b) $\alpha\ne 0$.} It follows from $\mathrm{(iii)}$, $\mathrm{(iv)}$, and $\mathrm{(v)}$ that
\begin{equation*}
\begin{split}
&g'(t)=2\mathrm{Re}\Big[\Big(\frac{1}{2}+\frac{F_{\tau}(z_2(t),\mathrm{Im}~z_1(t))}{2i}\Big)L(z_1(t))a(z_2(t))\\
&\quad +\Big(P_{z_2}(z_2(t))+F_{z_2}(z_2(t),\mathrm{Im}~z_1(t))\Big) i z_2(t)\Big]\\
&=2\mathrm{Re}\Big[\Big(\frac{1}{2}+\frac{F_{\tau}(z_2(t),\mathrm{Im}~z_1(t))}{2i}\Big)\frac{1}{\alpha}\Big(\exp\Big(\alpha\big(i\mathrm{Im}~z_1(t)+g(t)-P(z_2(t))\\
&\quad -F(z_2(t),\mathrm{Im}~z_1(t))\big)\Big)-1 \Big)a(z_2(t))+\Big(P_{z_2}(z_2(t))+F_{z_2}(z_2(t),\mathrm{Im}~z_1(t))\Big) i z_2(t)\Big]\\
&=2\mathrm{Re}\Big[\frac{1}{\alpha}\frac{i+F_{\tau}(z_2(t),\mathrm{Im}~z_1(t))}{2i}\exp\Big(\alpha\big(i\mathrm{Im}~z_1(t)-
F(z_2(t),\mathrm{Im}~z_1(t))\big)\Big)\\
&\quad \times \exp(-\alpha P(z_2(t)) )\exp(\alpha g(t)) a(z_2(t))  - \frac{1}{\alpha}\Big(\frac{1}{2}+\frac{F_{\tau}(z_2(t),\mathrm{Im}~z_1(t))}{2i}\Big)a(z_2(t)) \\
&\quad +\Big(P_{z_2}(z_2(t))+F_{z_2}(z_2(t),\mathrm{Im}~z_1(t))\Big) i z_2(t)\Big]\\
&=2\mathrm{Re}\Big[\frac{1}{\alpha}\frac{i+Q_0(z_2(t))}{2i}\exp(-\alpha P(z_2(t)) )\exp(\alpha g(t))a(z_2(t))\\
&\quad - \frac{1}{\alpha}\Big(\frac{1}{2}+\frac{F_{\tau}(z_2(t),\mathrm{Im}~z_1(t))}{2i}\Big)a(z_2(t)) +\Big(P_{z_2}(z_2(t))+F_{z_2}(z_2(t),\mathrm{Im}~z_1(t))\Big) i z_2(t)\Big]\\
&=2\mathrm{Re}\Big[i z_2(t)P_{z_2}(z_2(t))+\Big(\frac{1}{2}+\frac{Q_0(z_2(t))}{2i}\Big)\frac {\exp(-\alpha P(z_2(t)))-1}{\alpha}a(z_2(t)) \Big]\\
&\quad+2 \mathrm{Re}\Big[ i z_2(t) F_{z_2}(z_2(t),\mathrm{Im}~z_1(t))+ \frac{1}{2 \alpha}\Big(F_{\tau}(z_2(t),\mathrm{Im}~z_1(t))-Q_0(z_2(t))\Big)ia(z_2(t))\Big]\\
&\quad + 2\frac{\exp(\alpha g(t))-1}{\alpha} \exp(-\alpha P(z_2(t))) \mathrm{Re}\Big[\Big(\frac{1}{2}+\frac{Q_0(z_2(t))}{2i}\Big)a(z_2(t)) \Big] \\
&=2\frac{\exp(\alpha g(t))-1}{\alpha} \exp(-\alpha P(z_2(t))) \mathrm{Re}\Big[\Big(\frac{1}{2}+\frac{Q_0(z_2(t))}{2i}\Big)a(z_2(t)) \Big]
\end{split}
\end{equation*}
 for every $t\in \mathbb R$. Since $g(0)=\rho(p_0)=0$, again by the uniqueness of the solution of differential equations, we conclude that $g(t)\equiv 0$. This ends the proof. 
\end{proof}

\end{document}